\documentclass{amsart}
\usepackage{latexsym,amssymb,amsmath,amsfonts,amsthm, enumerate, tikz, color,float,enumitem, theoremref, dsfont, mathrsfs, pgfplots, subfigure}
\usepackage{comment}
\usepackage{mathtools}
\mathtoolsset{showonlyrefs}

\newtheorem{thm}{Theorem}

\newtheorem{lemma}[thm]{Lemma}
\newtheorem{remark}[thm]{Remark}
\newtheorem{prop}[thm]{Proposition}
\newtheorem{cor}[thm]{Corollary}
 
\newcommand{\f}{\frac}
\renewcommand{\P}{P_{\lambda,\rho} }

\newcommand{\T}{\mathbb{T}}

\newcommand{\RR}{\mathfrak{R}}
\newcommand{\R}{\mathcal R}
\newcommand{\B}{\mathcal B}
\newcommand{\Cp}{CED$^+$ }
\newcommand{\LM}{\lambda_c^-}
\newcommand{\LP}{\lambda_c^+}

\newcommand{\ckl}{C_k^{\lambda,\rho}}

\title{Chase-escape with death on trees}

\author[]{Erin Beckman}
\address{Department of Mathematics and Statistics, Concordia University}
\email{ebeckman@math.duke.edu}

\author[]{Keisha Cook}
\address{Department of Mathematics, Tulane University}
\email{kcook7@tulane.edu}

\author[]{Nicole Eikmeier}
\address{Department of Computer Science, Grinnell College}
\email{eikmeier@grinnell.edu}

\author[]{Sarai Hernandez-Torres}
\address{Department of Mathematics, University of British Columbia}
\email{saraiht@math.ubc.ca}

\author[]{Matthew Junge}
\address{Department of Mathematics, Bard College}
\email{mjunge@bard.edu}

\begin{document}

\begin{abstract}
Chase-escape is a competitive growth process in which red particles spread to adjacent uncolored sites, while blue particles overtake adjacent red particles. We introduce the variant in which red particles die and describe the phase diagram for the resulting process on infinite $d$-ary trees. A novel connection to weighted Catalan numbers makes it possible to characterize the critical behavior.
%The probability of coexistence is not monotonic. We address this complication via a novel connection to weighted Catalan numbers.
%  we are able to describe the phase diagram for particle survival. 
\end{abstract}

%\HOX{Make clear that death is new. -MJ}
%\HOX{Would the abstract be more accessible if we write it in terms of prey and predators? -SH Open to debate. My thinking was this makes it sound more serious mathematically and less like an applied paper. -MJ}

%\thanks{}

\maketitle

%\tableofcontents

\section{Introduction}

\emph{Chase-escape} (CE) is a model for predator-prey interactions in which expansion of predators relies on but also hinders the spread of prey. The spreading dynamics come from the Richardson growth model \cite{richardson}. Formally, the process takes place on a graph in which vertices are in one of the three states $\{w,r,b\}$. 
Adjacent vertices in states $(r,w)$ transition to $(r,r)$ according to a Poisson process with rate $\lambda$. Adjacent $(b,r)$ vertices transition to $(b,b)$ at rate $1$. 
The standard initial configuration has a single vertex in state $r$ with a vertex in state $b$ attached to it. All other sites are in state $w$. These dynamics can be thought of as prey $r$-particles ``escaping" to empty $w$-sites while being ``chased" and consumed by predator $b$-particles. We will refer to vertices in states $r,b,$ and $w$ as being red, blue, and white, respectively. % There has been recent interest \cite{cande,tang,khordzakia} in characterizing geometries and values of $\lambda$ for which red can ``escape" i.e., occupy infinitely many sites of the graph. 

We introduce the possibility that prey dies for reasons other than predation in a variant which we call \emph{chase-escape with death} (CED). This is CE with four states $\{ w,r,b,\dagger \}$ and the additional rule that vertices in state $r$ transition to state $\dagger$ at rate $\rho>0$. We call such vertices \emph{dead}. Dead sites cannot be reoccupied. 
%Since we will sometimes refer to chase-escape without death, CE refers to the process with $\rho=0$ and CED to the setting $\rho > 0$.

%Given a graph and initial configuration, CED is specified by the spreading rate and the death rate of red particles. 
%To keep our notation light, we will often suppress the dependence on $\lambda$ and $\rho$ by writing $P(\:\cdot\:)$ and $E[\:\cdot\:]$ in place of $\P(\:\cdot\:)$ and $E_{\lambda,\rho}[\:\cdot\:]$. 
\subsection{Results}
 
We study CED on the infinite rooted $d$-ary tree $\T_d$---the tree in which each vertex has $d\geq 2$ children---with an initial configuration that has the root red, one extra blue vertex $\mathfrak b$ attached to it, and the rest of the vertices white. 
Let $\R$ be the set of sites that are ever colored red. Similarly, let $\B$ be the set of sites that are ever colored blue. Denote the events that red and blue occupy infinitely many sites by
$A  = \{ |\R| = \infty\} \text{ and } B = \{ |\B| = \infty\}.$
Since $\B - \{\mathfrak b\} \subseteq \R$ deterministically, we also have $B \subseteq A$. %Given $d,\lambda$, and $\rho$, t
We will typically write $P$ and $E$ in place of $\P$ and $E_{\lambda,\rho}$ for probability and expectation when the rates are understood to be fixed. There are three possible phases for CED:
\begin{itemize}
	\item \emph{Coexistence} $P(B) > 0$.
	\item  \emph{Escape} $P(A)>0$ and $P(B) =0$.
	\item \emph{Extinction} $P(A) = 0$. 

\end{itemize} 
%\HOX{I suggest either pointing to Theorem 1 here, or emphasizing that this is not a given... i.e. In our analysis of CED we show that there are three distinct phases: -NE}

For each fixed $d$ and $\lambda$, we are interested in whether or not these phases occur and how the process transitions between them as $\rho$ is varied.
Accordingly, we define the critical values
\begin{align}
\rho_c &= \rho_c(d,\lambda)= \inf \{ \rho \colon P_{\lambda, \rho} (B) =0\} \label{eq:rho_c}, \\
\rho_e &= \rho_e(d,\lambda) = \inf \{ \rho \colon P_{\lambda, \rho} (A) = 0\}.% = \lambda(d-1). 	
\end{align}
%In words, $\rho_c$ is the infimum over all death rates at which coexistence occurs, and $\rho_e$ is the infimum over all death rates for which that extinction occurs.
 %The explicit formula for $\rho_e$ comes from an elementary branching process calculation in which we assume there are no blue particles.
%Note that, because $P(B) \leq P(A)$, we always have $\rho_c \leq \rho_e$. 

One feature of CE, and likewise CED, that makes it difficult to study on graphs with cycles is that there is no known coupling that proves $P(B)$ increases with $\lambda$. On trees the coupling is clear, which makes analysis more tractable. It follows from \cite{bordenave2014extinction} that $P(B) >0$ in CE on a $d$-ary tree if and only if $\lambda > \LM$ with 
\begin{align}
\LM = 2d -1 - 2\sqrt{d^2 -d} \sim (4d)^{-1}. \label{eq:LM}
\end{align}
For CED on $\mathbb T_d$, $P(B)$ is no longer monotonic with $\lambda$. As $\lambda$ increases, blue falls further behind and so the intermediate red particles must live longer for coexistence to occur.
This lack of monotonicity makes $$\LP  = 2d -1 + 2 \sqrt{ d^2 - d} \sim 4d$$
also relevant, because we will see that when $\lambda\geq \lambda_c^+$, the gap between red and blue is so large that coexistence is impossible for any $\rho >0$.

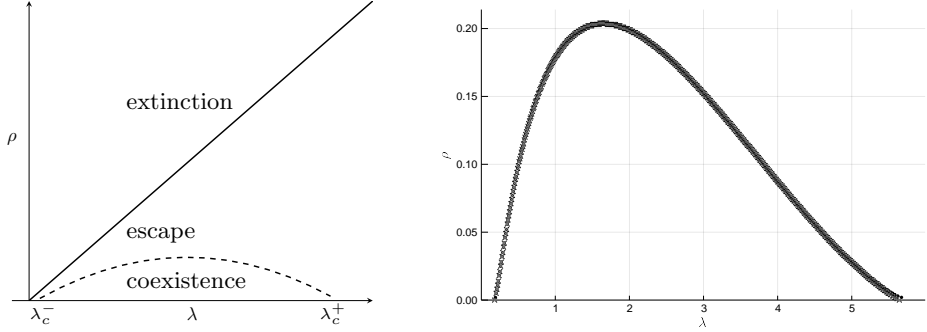
\begin{figure} 
%	\centering 
		\begin{tikzpicture}[scale = .7]
			   \begin{axis}[
			       axis y line = middle,
			       axis x line = middle,
			       ticks = none,
			       yticklabels={,,}
			       xticklabels={,,}
			       samples     = 200,
			       domain      = 0:20,
			       xmin = -1, xmax = 20,
			       ymin = 0, ymax = 80,
			       unbounded coords=jump,
			       x label style={at={(axis description cs:0.5,0)},anchor=north},
			       y label style={at={(axis description cs:0,.5)},anchor=south},
			       xlabel = {$\lambda$},
			       ylabel = {$\rho$}
			     ]
%			     \addplot[red, thick, mark=none] {(1/4)*( ( (8*5 + 2)*x + x^2 +1  )^(1/2) - 3*x - 3};
%				\addplot[black, thick, dashed, mark=none] {(1/4)*( ( (14*5 + 2)*x + x^2 +1  )^(1/2) - 3*x - 3}; 
			     \addplot[black, thick, dashed, mark=none, domain = -3:26] {(-(1/300)*(x- 18)^3*(x- .4 )^3)^(1/3)};
%				\addplot[black, thick,, dashed, mark=none] {4* (1/12)*( (( (96*5 +2)*x + x^2 +1  )^(1/2) - 5*x - 5)^(5/5)};
				\addplot[black, thick, mark=none] {x*(5-1)};
%				\addplot[black,thick, mark=none, domain = 18:22]{72};
			   \end{axis}
			   \node at (.6,-.25) {\tiny $\lambda_c^-$};
			   \node at (6.1,-.25) {\tiny $\lambda_c^+$};
%			   \node at (5,.9) {\small $\rho_c$};
%			   \node at (5,4.6) {\small $\rho_e$}; 
			   \node[anchor = west] at (2,3.8) {\small extinction};
			   \node[anchor = west] at (2,1.25) {\small escape};
			   \node[anchor = west] at (2,.4) {\small coexistence};
			\end{tikzpicture}
\qquad 				\includegraphics[width = 6.65 cm]{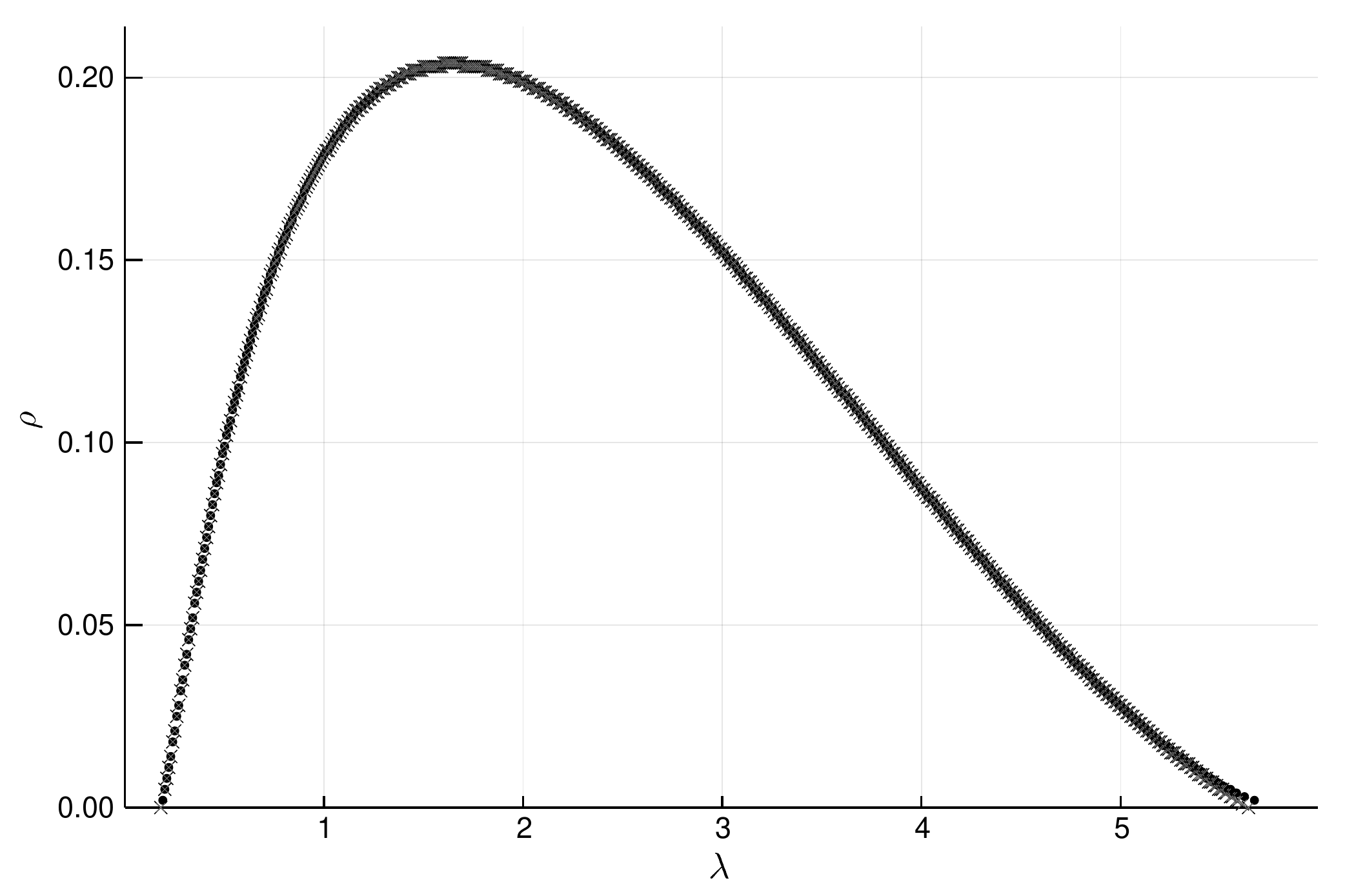}
		\caption{ (Left) The phase diagram for fixed $d$. The dashed line is $\rho_c$ and the solid line $\rho_e$. (Right) A rigorous approximation of $\rho_c$ when $d=2$. The approximations for larger $d$ have a similar shape.
		} \label{fig:pd}
\end{figure}

Suppressing the dependence on $d$, let $\Lambda = (\LM,\LP)$. Unless stated otherwise, we will assume that $d \geq 2$ is fixed. Our first result describes the phase structure of CED (see Figure~\ref{fig:pd}). 
\begin{thm}  \thlabel{thm:main} Fix  $\lambda > 0$. 
\begin{enumerate}[label = (\roman*)]
%	\item 	$\rho_c = \inf \{ \rho \colon F(\rho) \leq 1\}.$
%	\item  On the set $\{ \rho \colon P(B) >0\}$ we have $P(B)$ is a continuous, strictly decreasing function of $\rho$.	
	\item If $\lambda \in \Lambda$, then $0 < \rho_c < \rho_e=\lambda(d-1)$ with escape occurring at $\rho = \rho_c$, and extinction at $\rho = \rho_e$. 	
	\item If $\lambda \notin \Lambda$, then $0=\rho_c < \rho_e = \lambda(d-1)$ with extinction occurring at $\rho = \rho_e$.
\end{enumerate} 
\end{thm}

%Understanding $\rho_e$ comes from straightforward comparisons to the branching process with no blue particles present. 
Our next result concerns the behavior of $E|\B|$ at and above criticality.

 \begin{thm} \thlabel{cor:EB}
Fix $\lambda \in \Lambda$. 
\begin{enumerate}[label = (\roman*)]
	\item \label{>} If $\rho > \rho_c$, then $E |\B|  < \infty.$
	\item \label{=} If $\rho = \rho_c$, then $E |\B| = \infty$.
\end{enumerate}
\end{thm}

\thref{cor:EB} \ref{=} is particularly striking because it is known that $E |\B| < \infty $  in CE with $\lambda = \lambda_c^-$ (see \cite[Theorem 1.4]{bordenave2014extinction}). Hence the introduction of death changes the critical behavior. The reason for this comes down to singularity analysis of a generating function associated to CED and is discussed in more detail in \thref{rem:EZ}.

We prove three further results about $\rho_c$ concerning its asymptotic growth in $d$, smoothness in $\lambda$, and approximating its value (see Figure \ref{fig:pd}).

\begin{thm} \thlabel{cor:growth}
Fix $\lambda>0$, $c < \sqrt{\lambda /2}$, and $C > \sqrt{2 \lambda}$. For all $d$ large enough
$$c \sqrt{ d} \leq \rho_c \leq C\sqrt{d}.$$
\end{thm}

%We also prove that $\rho_c$ is smooth.
%The implicit function theorem ensures that $\rho_c$ is indeed smooth.

\begin{thm} \thlabel{thm:smooth}
The function $\rho_c$ is infinitely differentiable in $\lambda \in \Lambda$. 
\end{thm}

%Although we do not have a closed form, we provide an algorithm that
 %gives arbitrarily close approximations to $\rho_c$ . 
 %We do not estimate the complexity, but in practice the algorithm is quite fast. 

\begin{thm} \thlabel{cor:algorithm}
	For $\lambda\in \Lambda$ and $\rho \neq \rho_c$, there is a finite algorithm to determine if $\rho < \rho_c$ or if $\rho > \rho_c$. 
\end{thm}
%Note that the precision of the \thref{cor:algorithm} allows us to rigorously conclude that $P_{\lambda,\rho}(B)$ is not monotonic in $\lambda$. Given $d$, we can easily find values $\lambda_1< \lambda_2$ such that $0 < P(\lambda_2) < P(\lambda_1)$. See Figure \ref{fig:pd}.

%\begin{figure}
%\centering 
%	\includegraphics[width = 6 cm]{bounds_d=2.pdf}	\quad 	\includegraphics[width = 6 cm]{bounds_d_5.pdf}	
%	\caption{Rigorous approximations of $\rho_c(\lambda)$ when $d=2$ (left) and $d=5$ (right) using the algorithm from \thref{cor:algorithm}. 
%	%Larger values of $d$ exhibit similar shaped, but larger curves.
%	}\label{fig:curve}
%\end{figure}

\subsection{Proof methods} \label{sec:methods}

Theorems \ref{thm:main} and \ref{cor:growth} are proven by relating coexistence in CED to the survival of an embedded branching process that renews each time blue is within distance one of the farthest red. Describing the branching process comes down to understanding how CED behaves on the nonnegative integers with 0 initially blue and 1 initially red. 
In particular, we are interested in the event $\RR_k$ that at some time $k$ is blue, $k+1$ is red, and all further sites are white.

The probability of $\RR_k$ can be expressed as a weighted Catalan number. 
%In general, t
These are specified by non-negative weights $u(j)$ and $v(j)$ for $j\geq 0$. Given a lattice path $\gamma$ consisting of unit rise and fall steps, each rise step from $(x,j)$ to $(x+1,j+1)$ has weight $u(j)$, while fall steps from $(x,j+1)$ to $(x+1,j)$ has weight $v(j)$. The \emph{weight} $\omega(\gamma)$ of a Dyck path $\gamma$ is defined to be the product of the rise and fall step weights along $\gamma$. 

The corresponding \emph{weighted Catalan number} is 
$C_k^{u,v} = \sum \omega(\gamma)$
where the sum is over all Dyck paths $\gamma$ of length $2k$ (nonnegative paths starting at $(0,0)$ consisting of $k$ rise and $k$ fall steps). See Figure \ref{fig:weighted} for an example. For CED, we define $C^{\lambda,\rho}_k$ as the weighted Catalan number with weights
\begin{align}
u(j) =	\f{\lambda}{1+\lambda + (j+1) \rho} \text{ and } v(j) = \f{1}{1+\lambda + (j+2) \rho} \label{eq:ud}.
\end{align}
At \eqref{eq:R_k=C_k} we explain why $P(\RR_k) = C^{\lambda,\rho}_k$. 
 
Returning to CED on $\mathbb T_d$, self-similarity ensures that the expected number of renewals in the embedded branching process is equal to the generating function $g(z) = \sum_{k=0}^\infty C_k^{\lambda, \rho} z^k$ evaluated at $z=d$. We prove in \thref{cor:rho_d=rho_c} that $\rho_c$ is the value at which the radius of convergence of $g$ is equal to $d$. We characterize the radius of convergence using a  continued fraction representation of $g$, which leads to the proofs of Theorems \ref{cor:EB}, \ref{thm:smooth}, and \ref{cor:algorithm}.

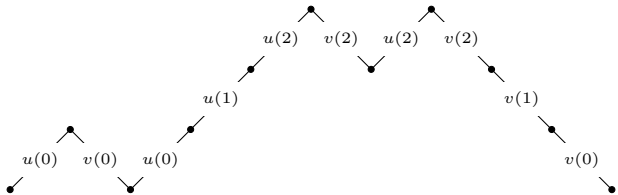
\begin{figure}
	\begin{tikzpicture}[scale = .8]
			%\draw (0,3) -- (0,0) -- (10,0);
			\draw (0,0) -- (1,1) node [midway,fill=white] {\tiny $u(0)$} -- (2,0)  node [midway,fill=white] {\tiny $v(0)$} -- (3,1) node [midway,fill=white] {\tiny $u(0)$} -- (4,2) node [midway,fill=white] {\tiny $u(1)$}-- (5,3) node [midway,fill=white] {\tiny $u(2)$} -- (6,2) node [midway,fill=white] {\tiny $v(2)$} --  (7,3) node [midway,fill=white] {\tiny $u(2)$} -- (8,2) node [midway,fill=white] {\tiny $v(2)$} -- (9,1) node [midway,fill=white] {\tiny $v(1)$} -- (10,0) node [midway,fill=white] {\tiny $v(0)$};
			
	\draw[fill] (0,0)  circle [radius=0.05]
            (1,1)  circle [radius=0.05]
            (2,0) circle [radius=0.05]
            (3,1) circle [radius=0.05]
            (4,2) circle [radius=0.05]
            (5,3) circle [radius=0.05]
            (6,2) circle [radius=0.05]
            (7,3) circle [radius=0.05]
            (8,2) circle [radius=0.05]
            (9,1) circle [radius=0.05]
            (10,0) circle [radius=0.05]
            ;

	\end{tikzpicture}

\caption{A Dyck path of length $10$ with weight $u(0)^2 v(0)^2 u(1)v(1) u(2)^2 v(2)^2.$}\label{fig:weighted}
\end{figure}

\subsection{History and context} \label{sec:history}
The forebearer of CE is the Richardson growth model for the spread of a single species \cite{richardson}. In our notation, this process corresponds to the setting with $\lambda =1$, $\rho=0$, and no blue particles. 
%This was introduced on the two-dimensional lattice, but can be studied on any graph. 
Many basic questions remain unanswered for the Richardson growth model on the integer lattice \cite{fpp}, as well as for the competitive version \cite{pleasures}.

James Martin conjectured that coexistence occurs in CE on lattices when red spreads at a slower rate than blue. Simulation evidence from  Tang, Kordzakhia, Lalley in \cite{si} suggested that, on the two-dimensional lattice, red and blue coexist with positive probability so long as $\lambda \geq 1/2$. Durrett, Junge, and Tang proved in \cite{DJT} that red and blue can coexist with red stochastically slower than blue on high-dimensional oriented lattices with spreading times that resemble Bernoulli bond percolation. 
 
The first rigorous result we know of for CE is Kordzakhia's proof that the phase transition occurs at $\lambda_c^-$ for CE on regular trees \cite{tree1}. Later, Kortchemski considered the process on the complete graph as well as trees with arbitrary branching number \cite{complete, tree_chase}. An alternative perspective of CE as scotching a rumor was studied by Bordenave in \cite{rumor}. 
The continuous limit of rumor scotching was studied many years earlier by Aldous and Krebs \cite{aldous}. Looking to model malware spread and suppression through a device network, Hinsen, Jahnel, Cali, and Wary studied CE on Gilbert graphs \cite{gilbert}.

To the best of our knowledge, CED has not been studied before. From the perspective of modeling species competition, it seems natural for prey to die from causes other than being consumed, and, in rumor scotching, for holders to cease spread because of fading interest. Furthermore, CED has new mathematical features. The existence of an escape phase, the fact that $E|\B| = \infty$ at criticality, and the lack of monotonicity of $P(B)$ in $\lambda$ are all different than what occurs in CE.

The perspective we take on weighted Catalan numbers also appears to be novel. Typically they are studied with integer weights \cite{combo1, combo3, shaderweighted}. We are interested in the fractional weights at \eqref{eq:ud}. Flajolet and Guilleman observed that fractionally weighted lattice paths describe birth and death chains for which the rates depend on the population size \cite{birth_death}. The distance between the rightmost red and rightmost blue for CED on the nonnegative integers is a birth and death chain in which mass extinction may occur. Since we are interested in CED on trees, we analyze the radius of convergence of the generating function of weighted Catalan numbers, which we believe has not been studied before. 

%This appears to be a new perspective. 

%\subsection{Organization}
%In Section \ref{sec:CED+} we describe CED on the nonnegative integers and some of its properties. Section \ref{sec:cat} elaborates on properties of the generating function for the weighted Catalan numbers. Section \ref{sec:M} then describes the phase structure for CED on the tree in terms of the generating function's radius of convergence. Section \ref{sec:proofs} has the proofs of \thref{thm:main}, \thref{cor:EB}, and \thref{cor:growth}. Section \ref{sec:smooth} proves \thref{thm:smooth} and Section \ref{sec:approximations} describes the algorithm in \thref{cor:algorithm}. 

\section{CED on the line} \label{sec:CED+}

%We start with a formal construction of some important quantities in the one-dimensional setting. 
Let \Cp denote the process with CED dynamics on the nonnegative integers for which the vertices $0$ and $1$ are initially blue and red, respectively. All other vertices are initially white. 
Let $s_t (n) \in \{ w,r,b,\dagger \}$ indicate the state of vertex $n$ at time $t$. Define the processes
$B_t = \max \{ n \colon s_t(n) = b \}$, $R_t = \max \{ n \colon s_t(n) = r \}$, and the random variable
\begin{align}
Y = \sup \{ B_t \colon t \geq 0 \}. \label{eq:Ydef}
\end{align}
If $s_t(n) \neq r$ for all $n$, then define $R_t = -\infty$. 
Let $\partial_t = (s_t(n))_{n = B_t}^{R_t} $ 
be the state of the interval $[B_t, R_t]$. One can think of this as the boundary of the process. Note that this interval only makes sense when $R_t > -\infty$. \emph{Renewal times} are times $t \geq 0 $ such that
$\partial_t = (b,r).$
For $k \geq 0$ let
\begin{align}
\RR_k = \{B_t = k \text{ for some renewal time $t$}\}\label{eq:def.R_k}
\end{align}
be the event that there is a renewal when blue occupies site $k$. Also define the event $A_t = \lbrace s_t (n) \neq \dagger \text{ for all } n\}$ that none of the red sites have died up to time $t$.

%The goal of this section is to understand $F(\rho)$ from the introduction. 
Let $S_t = R_t - B_t$ be the distance between the rightmost blue and red particles at time $t$. 
Define the collection of \emph{jump times} as $\tau (0) = 0$ and for $i \geq 1$
 $$\tau (i) = \inf \{ t \geq \tau (i-1) \colon S_t \neq S_{ \tau (i-1) } \text{ or } \mathds{1} (A_{t} ) = 0   \}.$$ 
The \emph{jump chain} $J = (J_i)$ of $S_t$ is given by
\begin{equation*}
	J_i =
	\begin{cases}
	S_{ \tau (i) },  & \mathds{1} (A_{\tau(i)}) = 1 \\
	0       ,        & \text{otherwise}
	\end{cases}.
\end{equation*}
This is a Markov chain with an absorbing state $0$ corresponding to blue no longer having the potential to reach infinity. The transition probabilities for $j > 0$ are:
\begin{align}
	p_{j, j+1} 
	&= \frac{\lambda}{1+\lambda+j\rho}, 
	\qquad
	p_{j, j -1}
	 = 
	\frac{1}{1+\lambda+j\rho}, \qquad 
	p_{j, 0}
	 = 
	\frac{j\rho}{1+\lambda+j\rho}. \label{eq:transitions}
\end{align}

Call a jump chain $(J_0,\hdots, J_n)$ \emph{living} if $J_i>0$ for all $0 \leq i \leq n$.  
Translating the set of Dyck paths of length $2k$ up by one vertical unit gives the jump chains corresponding to $\RR_k$. Notice  that  $p_{j,j+1} = u(j-1)$ and $p_{j,j-1} = v(j-2)$ with $u$ and $v$ as in \eqref{eq:ud}. Thus, it is easy to see that for all $k\geq 0$ we have
\begin{align}
P_{\lambda,\rho}(\RR_k) = C_k^{\lambda,\rho},	\label{eq:R_k=C_k}
\end{align}
with $C_k^{\lambda,\rho}$ the weighted Catalan number defined in Section \ref{sec:methods}.

% We conclude this section with two combinatorial lemmas. The first gives a lower bound on $C_k^{\lambda, \rho}$, the second bounds $P(Y\geq k)$ in terms of  $P(\RR_k)$.

\begin{lemma} \thlabel{lem:Rlb}
For any $\epsilon>0$ there exists $\rho'>0$ such that for all $\rho \in [0,\rho')$ and sufficiently large $k$ $$C_k^{\lambda,\rho} \geq \left(\f{(4-\epsilon)\lambda}{(1+\lambda)^2} \right)^k.$$ 
\end{lemma}

\begin{proof}
Let $C_{k,m}$ be number of Dyck paths of length $2k$ that never exceed height $m$. We first claim that for any $\delta >0$, there exists $m_\delta$ such that 
\begin{align}C_{k,m_\delta} \geq (4-\delta)^k \label{eq:Ckm}\end{align}
for sufficiently large $k$. 

The $m$th Catalan number $C_m := C_{m,\infty}$ counts the number of Dyck paths of length $2m$. Consider any sequence of $\lfloor k/m\rfloor$ Dyck paths of length $2m$. If we concatenate these paths, we have a path of length $2m\lfloor k/m \rfloor \geq 2k-2m$ which stays below height $m$. We extend this path into a Dyck path of length $2k$ by concatenating the necessary number of up and down steps to the end in any manner. 
Since each of the $\lfloor k/m\rfloor$ Dyck paths of length $2m$ can be chosen independently of one another, we have $C_{k,m} \geq (C_{m})^{\lfloor k/m\rfloor}.$

Using the standard asymptotic relation $C_{m} \sim (1/\sqrt \pi) m^{-3/2} 4^{m}$ (see \cite{stanley15}), we have for large enough $m,k$
$$C_{k,m} \geq  \left(\f{4^{m}}{2\sqrt{\pi} m^{3/2}} \right)^{\lfloor k/m\rfloor} \geq \left( \f{4^{m}}{2\sqrt{\pi} m^{3/2} } \right)^{-1} \left( \f {4}{(2\sqrt{\pi} m^{3/2} )^{1/m} }\right)^{k}.$$

It is easy to verify that $(2\sqrt{\pi} m ^{3/2} )^{1/m} \to 1$ as $m \to \infty$. Thus, we can choose $m$ large enough so that 
$$\f{4}{( 2\sqrt{ \pi} m^{3/2})^{1/m}} > 4-  \f \delta 2.$$
We then have 
$$C_{k,m} \geq C\left( \f{4^{m}}{2\sqrt{\pi} m^{3/2} } \right)^{-1} (4-(\delta/2))^{k}.$$
This is true for all $m, k$ sufficiently large, and we can see that if we fix an $m_\delta$ big enough so that this inequality holds, then we can increase $k$ enough such that we have the claimed inequality at \eqref{eq:Ckm}.

Using the weights at \eqref{eq:ud}, each path $\gamma$ counted by $C_{k,m}$ satisfies
$$\omega(\gamma) \geq \left( \f{\lambda }{(1+ \lambda + m \rho)^2}\right)^{k},$$
because $\gamma$ has length $2k$ but never exceeds height $m$. Summing over just the Dyck paths counted by $C_{k,m}$ gives
\begin{align}
C_k^{\lambda,\rho} \geq C_{k,m} \f{\lambda^k }{(1+ \lambda + m \rho)^{2k}}.\label{eq:Ckmlb}	
\end{align}

Fix an $\epsilon > 0$. We can choose $\delta>0$ small enough so that $(4-\delta)(1-\delta) > 4-\epsilon$. By \eqref{eq:Ckm}, pick $m_\delta$ large enough to have $C_{k,m_\delta} > (4-\delta)^k$ for all sufficiently large $k$. Finally, choose $\rho'>0$ small enough so that 
$$\f{\lambda}{(1+ \lambda + m_\delta \rho')^2} > \f{\lambda}{(1+\lambda)^2}(1-\delta).$$
Since the $C_k^{\lambda,\rho}$ are decreasing in $\rho$, applying these choices to \eqref{eq:Ckmlb} gives the desired inequality for all $\rho \in [0 , \rho')$:
$$C_{k}^{\lambda,\rho} \geq \left( (4- \delta)(1-\delta)\f{\lambda}{(1+\lambda)^2}\right)^k \geq \left(\f{(4-\epsilon)\lambda}{(1+\lambda)^2} \right)^k.$$ 
\end{proof}

Recall the definition of $Y$ at \eqref{eq:Ydef}. We conclude this section by proving that $P(Y \geq k)$ can be bounded in terms of $P(\RR_k)$. The difficulty is that the event $\{Y\geq k\}$ includes \emph{all} realizations for which blue reaches $k$, while $\RR_k$ only includes realizations which have a renewal at $k$.  

\begin{lemma} \thlabel{lem:bigger}
	For $\rho >0$, there exists $C>0$, which is a function of $\lambda, \rho$, such that $P(Y\geq k) \leq C k^{1+\lambda/\rho} P(\RR_k)$ for all $k \geq 1$.
\end{lemma}

\begin{proof}

Given a living jump chain $J = (J_0,J_1,\hdots, J_m)$, define the \emph{height profile} of $J$ to be 
$h(J) = (h_1(J),\hdots, h_{m+1}(J)),$
where $h_i(J)$ are the number of entries $J_\ell$ in $J$ with $\ell < m$ for which $J_\ell =i$. These values correspond to the total number of times that blue is at distance $i$ from red. Suppose that red takes $r(J)$ many steps in a jump chain $J$. 
 It is straightforward to show that
\begin{align}
p(J) = \lambda^{r(J)} \prod_{j=1}^{m+1} \left(\f{1}{1+ \lambda + j\rho} \right)^{h_j(J)}	\label{eq:gen_p}
\end{align}
is the probability that the process follows the living jump chain $J$.
%Note that since $J_0 =1$, the maximum height that can be reached in $m$ steps is $m+1$. 

We view realizations leading to outcomes in $\{Y \geq k\}$ as having two distinct stages. In the \emph{first stage}, the rightmost red particle reaches $k$. In the \emph{second stage}, we ignore red and only require that blue advances until it reaches $k$. The advantage of this perspective is that we can partition outcomes in $\{Y \geq k\}$ by their behavior in the first stage and then restrict our focus to the behavior of the process in the interval $[0,k]$ for the second stage.

 For $2 \leq \ell \leq k$, define $\Gamma_\ell$ to be the set of all living jump chains of length $2k-\ell-1$ which go from $(0,1)$ to $(2k-\ell-1, \ell)$ with the last step being a rise step (see Figure \ref{fig:path}). These are the jump chains from the first stage.
Now we describe the second stage. Assume that blue is at  $k - \ell$ when red reaches $k$. For blue  to reach $k$, the red sites in $[k-\ell+1, k]$ must stay alive long enough for blue to advance another $\ell$ steps. This has probability
$
  \sigma(\ell) := \prod_{i=1}^\ell (1 + i \rho)^{-1}
$.
Given $\gamma \in \Gamma_\ell$, the formula at \eqref{eq:gen_p} implies that the probability $Y \geq k$ and the first $2k - \ell-1$ steps of the process jump chain follow $\gamma$ is
\begin{align}q(\gamma) := p(\gamma) \sigma (\ell) = \lambda^{k-1} \sigma (\ell) \prod_{j=1}^{k+1} (1+ \lambda + j \rho)^{-h_j(\gamma)}. \label{eq:q}\end{align}
%As we have partitioned the event $\{Y \geq k\}$, we have $P(Y \geq k) = \sum_{\ell = 2}^{k} \sum_{\gamma \in \Gamma_\ell} q(\gamma).$

Let $q_\ell = \sum_{\gamma \in \Gamma_\ell} q(\gamma)$. Notice that 
\begin{align}q_2 \f{\lambda}{(1+ \lambda + \rho)(1+ \lambda + 2\rho)^2} \leq P(\RR_k).\label{eq:qR}	
\end{align}
This is because a subset of $\RR_k$ is the collection of processes which follow jump chains in $\Gamma_2$ and for which the next three steps have blue advance by one, then red advance by one, followed by blue advancing one. 
We will further prove that there exists $C_0$ (independent of $\ell$) such that 
\begin{align}q_\ell \leq C_0 \ell^{\lambda/\rho} q_2.\label{eq:q_ineq}\end{align} The claimed inequality then follows from \eqref{eq:qR}, \eqref{eq:q_ineq}, and the partitioning $P(Y \geq k) = \sum_{\ell = 2}^k q_\ell$. 

To prove \eqref{eq:q_ineq}, notice that by inserting $\ell-2$ fall steps right before the final upward step of each path in $\gamma\in \Gamma_\ell$, we obtain a path $\tilde \gamma \in \Gamma_2$ (see Figure \ref{fig:path}). Since the paths $\gamma$ and $\tilde \gamma$ agree for the first $2k - \ell -2$ steps, we have from \eqref{eq:q}
\begin{align}
q(\gamma) &= \f{ \sigma (\ell)}{ \sigma (2)\prod_{i=1}^{\ell-2} (1+ \lambda + i \rho)^{-1}} q(\tilde \gamma) \\
&= \f{(1+ \lambda + \rho)(1+\lambda+2\rho)}{(1+\ell \rho)(1+(\ell-1)\rho)} q(\tilde \gamma) \prod_{i=3}^{\ell-2}\left( 1 + \f{ \lambda}{1 + i \rho} \right). 
\end{align}
Rewriting as a sum and using integral bounds, one can verify that 
$$\prod_{i=3}^{\ell-2}\left( 1 + \f{ \lambda}{1 + i \rho} \right) \leq C_1 \ell^{\lambda/\rho},$$ 
for some $C_1$ that depends on $\lambda$ and $\rho$. This gives $q(\gamma) \leq C_0\ell^{\lambda/\rho} q(\tilde \gamma)$. Thus, $q_\ell \leq \sum_{\gamma \in \Gamma_\ell} C_0 \ell^{\lambda/\rho} q(\tilde \gamma).$
If we restrict to paths $\gamma \in \Gamma_\ell$, the map $\gamma \mapsto \tilde \gamma$ is injective and hence
$$ q_\ell \leq \sum_{\gamma \in \Gamma_\ell} C_0 \ell^{\lambda/\rho} q(\tilde \gamma) \leq C_0 \ell^{\lambda/\rho} \sum_{\gamma \in \Gamma_2} q(\gamma) =  C_0 \ell^{\lambda/\rho} q_2.$$
This yields \eqref{eq:q_ineq} and completes the lemma.
\end{proof}
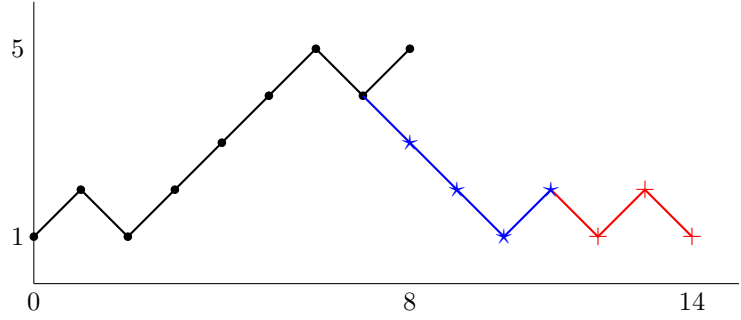
\begin{figure} 
\begin{center}
\begin{tikzpicture}[scale = 1.25]

% horizontal axis
\draw (0,0) -- (7.5,0);
% x-labels
\draw	(0,0)   node[anchor=north] {0}
%		(2,0)   node[anchor=north] {4}
		(4,0)   node[anchor=north] {8}
		(7,0) node[anchor=north] {14};
% y-labels
\draw	(0,.5)  node[anchor=east] {1}
%		(0,1.5)  node[anchor=east] {3}
		(0,2.5) node[anchor=east] {5};
% vertical axis
\draw (0,0) -- (0,3);
% J_i
\draw[thick] (0,.5) -- (.5,1) -- (1,.5) -- (1.5,1) -- (2,1.5) --(2.5, 2) -- (3,2.5) -- (3.5, 2) -- (4,2.5);
\draw[fill] (0,.5)  circle [radius=0.04]
            (.5,1)  circle [radius=0.04]
            (1,.5) circle [radius=0.04]
            (1.5,1) circle [radius=0.04]
            (2,1.5)  circle [radius=0.04]
            (2.5,2) circle [radius=0.04]
            (3,2.5) circle [radius=0.04]
            (3.5,2) circle [radius=0.04]
			(4,2.5) circle [radius=0.04];
%%\tilde \gamma						

\draw[thick, blue]  (3.5,2) -- (5,.5) -- (5.5,1);
%\draw (3.5,2)  circle [radius=0.05]
%            (4,1.5)  circle [radius=0.05]
%            (4.5,1)  circle [radius=0.05]
%            (5,.5)  circle [radius=0.05]
%			(5.5,1)  circle [radius=0.05];

%\draw[thick] (5,.5) -- (5.5,1);
%%tilde gamma
\draw[thick, red]  (5.5,1) -- (6,.5) -- (6.5,1) -- (7,.5);
\node[blue] at (5.5,1) { \Large $\star$};
\node[blue] at (5,.5) {\Large $\star$};
\node[blue] at (4.5,1) {\Large $\star$};
\node[blue] at (4,1.5) {\Large $\star$};
%R_k
\node[red] at (6,.5) {$+$};
\node[red] at (6.5,1) {$+$};
\node[red] at (7,.5) {$+$};
\end{tikzpicture}
\caption{
	Let $k=7$. The black line with dots is a path $\gamma \in \Gamma_{5}$. The blue line with stars is the modified path $\tilde \gamma \in \Gamma_2$. The red line with pluses is the extension of $\tilde \gamma$ to a jump chain in $\RR_7$.  
} \label{fig:path}
\end{center}
\end{figure}

%\HOX{Should we make a figure for the stages and for $\gamma \mapsto \tilde \gamma$? -MJ}

\section{Properties of weighted Catalan numbers}\label{sec:cat}

\subsection{Preliminaries}
Given a sequence $(c_n)_{n \geq 0}$, define the formal continued fraction
\begin{align}
K[c_0,c_1,\hdots]  := \cfrac{ c_0}{ 1- \cfrac{c_1}{1-  \ddots}}.\label{eq:K}
\end{align}
The $c_i$ may be fixed numbers, or possibly functions. Also, whenever we write $x$ we mean a nonnegative real number, and $z$ represents an arbitrary complex number.

The discussion in \cite[Chapter 5]{goulden1985combinatorial} tells us that, for general weighted Catalan numbers, $g(z):=\sum_{k=0}^\infty C_k^{u,v} z^k$ is equal to the function
\begin{align}f(z) := K[1,a_0z,a_1z,\hdots] \label{eq:f}\end{align}
for all $|z| < M$, where $M$ is the radius of convergence of $g$ centered at the origin, and $a_j= u(j)v(j)$. $M$ is the modulus of the nearest singularity of $g$ to the origin, or by the Hadamard-Cauchy theorem
\begin{align}
M 	&= \frac{1} {\limsup_{k \to \infty} (C_k^{u,v})^{1/k} }.\label{eq:M_limsup}
\end{align}
%An illustrative example to keep in mind are the usual Catalan numbers for which the generating function has radius of convergence $1/4$ with continued fraction representation $K[1,z,z,\hdots]$.
Let $u(j)$ and $v(j)$ be as in \eqref{eq:ud} so that, unless stated otherwise, we have
\begin{align}
a_j = \f{ \lambda}{(1+ \lambda + (j+1) \rho)(1+ \lambda + (j+2) \rho)}.	\label{eq:aj}
\end{align}

\subsection{Properties of $f$ and $M$}

Our proof that $f$ is meromorphic relies on a classical convergence criteria for continued fractions \cite[Theorem 10.1]{wall2018analytic}.

\begin{thm}[Worpitzky Circle Theorem]\thlabel{thm:worpitzky}
	Let $c_j\colon D \to \{|w| < 1/4\}$ be a family of analytic functions over a domain $D \subseteq \mathbb C^n$. Then
	$K[1,c_0(z),c_1(z),\hdots]$ 
	converges uniformly for $z$ in any compact subset of $D$ and the limit function takes values in $\{ |z-4/3| \leq 2/3\}$.
\end{thm}

\begin{cor} \thlabel{lem:meromorphic}
If $\rho>0$, then $f$ is a meromorphic function on $\mathbb C$.
\end{cor}

\begin{proof}
%Let $\Delta = \{ z \colon |z| \leq M\}$. 
We will prove that $f$ is meromorphic for all $z \in \Delta = \{ |z| < r_0\}$ with $r_0>0$ arbitrary.  Let $h_j(z) := K[a_jz, a_{j+1} z, \hdots]$
be the tail of the continued fraction so that $f(z) = K[1,a_0z, \hdots, a_{j-1}z, h_j(z)]$.  Since $\rho>0$ we have $|a_j|\downarrow 0$ as $j \to \infty$.
 It follows that for some $j= j(r_0)$ large enough, $|a_k z| \leq 1/4$ for all $k \geq j$ and $z\in \Delta$. \thref{thm:worpitzky} ensures that $|h_j(z)| < \infty$ and the partial continued fractions $K[a_j z, \hdots, a_n z]$
are analytic (again by \thref{thm:worpitzky}) and converge uniformly to $h_j$ for $z \in \Delta$.
 Thus, $h_j$ is a uniform limit of analytic functions and is therefore analytic on $\Delta$. 
 We can then write
$f(z) = K[1,a_0z , \hdots, a_{j-1} z, h_j(z)].$
Since each $a_i z$ is a linear function in $z$, $f$ is a quotient of two analytic functions.
%\HOX{Is it true that a finite continued fraction ($K[1,a_0z, \hdots, a_{j-1}z]$) is analytic on $\Delta$?  Or should it be analytic on $  \{ z \colon |z| < M\} $? Or meromorphic on $\Delta$? -SH I think $f$ is meromorphic on the whole complex plane. The proof reflects this upgrade. }
%As this holds at all $z_0$, we have $f$ is meromorphic on all of $\mathbb C$. 
\end{proof}

Next we show that the exponential growth rate of the $C_k^{\lambda,\rho}$ responds continuously to changes in $\rho$.

\begin{lemma} \thlabel{lem:cat_close}
Fix $\lambda>0$.

\begin{enumerate}[label = (\roman*)]
	\item Fix $\rho >0$. Given $0\leq \rho' <\rho$, there exists $\epsilon >0$ such that
\begin{align}\ckl (1+\epsilon)^k \leq C_k^{\lambda, \rho'}.  \label{eq:C1} \end{align}
	\item Fix $\rho >0$. Given $\epsilon ' >0$, there exists a $\delta > 0$ such that for $\rho' \in (\rho-\delta,\rho+\delta)$ 
\begin{align} (1- \epsilon')^k \leq \frac{C_k^{\lambda,\rho}}{C_k^{\lambda,\rho'}}\leq (1+\epsilon')^k \label{eq:C2}.
\end{align}
	\item If $\rho =0$, then given $\epsilon'>0$, there exists a $\delta>0$ such that for $\rho' \in [0, \delta)$ it holds that
\begin{align} 1 \leq \frac{C_k^{\lambda,\rho}}{C_k^{\lambda,\rho'}}\leq (1+\epsilon')^k \label{eq:C2}.
\end{align}
\end{enumerate}
	
\end{lemma}

\begin{proof}
Let $\gamma$ be a Dyck path of length $2k$. From the definition of $\ckl$, we have $w(\gamma)$ is a product of some combination of exactly $k$ of the $a_j$ terms. 
%Recall that 
%$$a_j= \frac{\lambda}{ (1+ \lambda + (j+1) \rho )(1+ \lambda +(j+2)\rho)}.$$
Let $a_j'$ be the weights corresponding to $\rho'$. It is a basic calculus exercise to show that, when $\rho' <\rho$ we have the ratio $a_j'/a_j>1$ is an increasing function in $j$. Let $a_0'/a_0 = 1+\epsilon$. Using the fact that $w(\gamma)$ and $w'(\gamma)$ have the same number of each weight, we can directly compare their ratio using the worst case lower bound 

\begin{align}
(1+\epsilon)^{k}= (a_0'/a_0)^{k}   \leq \frac{w'(\gamma)}{w(\gamma)}. \label{eq:w_ratio}
\end{align}
Cross-multiplying then summing over all paths $\gamma$ gives (i). 

Towards (ii), suppose $0 < \rho' < \rho$. Then we have $1 \leq a_j'/a_j \leq \left(\rho/\rho'\right)^2$.   Choose $\delta_1$ such that $(\rho-\delta_1)/\rho = \sqrt{1 - \epsilon'}$. Using the same logic as above, we have for $\rho' > \rho - \delta_1$ that
$$(1- \epsilon')^k \leq \left(\frac{\rho'}{\rho}\right)^{2k} \leq \frac{w(\gamma)}{w'(\gamma)} \leq 1^k \leq (1+ \epsilon)^k.$$
%
%$$\left(\frac{\rho-\delta_1}{\rho}\right)^{2k} \leq \frac{w(\gamma)}{w'(\gamma)} \leq (1+\epsilon')^k$$
%
%$$\left(1-\epsilon'\right)^{k} \leq \frac{w(\gamma)}{w'(\gamma)} \leq (1+\epsilon')^k$$
Now suppose $0 < \rho < \rho'$. Choose $\delta_2$ such that $(\rho + \delta_2)/\rho = \sqrt{1+\epsilon'}$. Then, following the same steps as above, we see that for $\rho' < \rho + \delta_2$

$$(1- \epsilon')^k \leq \left(1\right)^{k} \leq \frac{w(\gamma)}{w'(\gamma)} \leq \left(\frac{\rho'}{\rho}\right)^{2k}\leq (1+ \epsilon')^k.$$
%
%$$\left(1 - \epsilon'\right)^{k} \leq \frac{w(\gamma)}{w'(\gamma)} \leq \left(\frac{\rho+\delta_2}{\rho}\right)^{2k}$$
%
%$$\left(1 - \epsilon'\right)^{k} \leq \frac{w(\gamma)}{w'(\gamma)} \leq \left(1+\epsilon'\right)^{k}$$
The same reasoning as with \eqref{eq:w_ratio} and choosing $\delta = \min(\delta_1, \delta_2)$ gives (ii).

We now prove (iii). Notice that $C_{k}^{\lambda,0} = C_k (\lambda /(1+\lambda)^2)^k$ with $C_k$ the $k$th Catalan number, since the transition probabilities at \eqref{eq:transitions} for the associated jump chain are homogeneous when $\rho=0$. By \thref{lem:Rlb}, we have for a fixed $\tilde{\epsilon} > 0$ and the bound $C_k \leq 4^k$, there exists a $\rho' > 0$ such that  
\begin{align}
1 
\leq \frac{C^{\lambda,0}_k}{C^{\lambda,\rho'}_k} 
\leq \frac{C_k(\lambda/(1+\lambda)^2)^k}{(4-\tilde{\epsilon})^k \left( \lambda/(1+\lambda)^{2}\right)^{k}}  
= \f{ C_k}{ (4 - \tilde \epsilon)^k} \leq \left( \f { 4}{4 - \tilde \epsilon} \right)^k.\label{eq:long_ineq}\end{align}
Choosing $\tilde{\epsilon}$ such that $4/(4-\tilde{\epsilon}) = 1 + \epsilon'$ and letting $\delta$ be small enough so that \eqref{eq:long_ineq} holds for all $\rho' < \delta$, we obtain an identical bound to the one in \thref{lem:cat_close} (ii) but for $\rho'\in [0,\delta)$. 
%Continuity follows from a similar argument as used to prove continuity for $\rho >0$. 

\end{proof}

\begin{lemma} \thlabel{prop:M_cont}
$M$ is a continuous strictly increasing function for $\rho\in[0,\infty)$ satisfying $M(0) = (1+\lambda)^2/(4\lambda)$. 
\end{lemma}
\begin{proof}
First note that for any $\rho$, $M(\rho) \leq (a_0)^{-1}< \infty$. This is because $\ckl \geq \omega(\gamma_k) = (a_0)^k$ where $\gamma_k$ is the Dyck path consisting of $k$ alternating rise and fall steps. 

That $M$ is increasing follows immediately from the fact that the $C_k^{\lambda,\rho}$ are decreasing in $\rho$. To see that $M$ is strictly increasing we use \thref{lem:cat_close} (i) in combination with the definition of $M$ at \eqref{eq:M_limsup}. Indeed, the lemma implies that for any $0\leq \rho' < \rho$ there exists a $\delta >0$ such that
$$M(\rho') = \frac{1}{\limsup_{k \to \infty} (C_k^{\lambda,\rho'})^{1/k}} \leq \frac{1}{(1+\delta) \limsup_{k \to \infty} (C_k^{\lambda,\rho})^{1/k} } = \f{M(\rho)}{1 + \delta}.$$
So $M$ is strictly increasing. 
%Fix $\rho$ and $\epsilon >0$. We will show that there is a neighborhood of $\rho$ such that $|M(\rho') - M| < \epsilon$ for all $|\rho - \rho'| < \delta.$ 

%\HOX{Do you need the fact that $M$ is increasing of the display inequality $ M(\rho) \leq (1 + \epsilon) M(\rho') $? -SH  I guess not until the next step. - MJ}

To show continuity for $\rho >0$, we use \eqref{eq:M_limsup}. Fix $\rho > 0$ and $\epsilon > 0$. Let $\epsilon' =  \epsilon/M(\rho)$ and let $\delta > 0$ be as in \thref{lem:cat_close} (ii). For $\rho' \in (\rho - \delta, \rho + \delta)$ we have  

$$
\liminf_{k \to \infty} \left(\frac{C_k^{\lambda,\rho}}{C_k^{\lambda,\rho'}}\right)^{1/k} \leq  \frac{M(\rho')}{M(\rho)} \leq \limsup_{k \to \infty} \left(\frac{C_k^{\lambda,\rho}}{C_k^{\lambda,\rho'}}\right)^{1/k}. 
$$
%
%Properties of limsup and liminf show that 
%%
%$$ \liminf_{k\to\infty} \left(\frac{C_k^{\lambda,\rho}}{C_k^{\lambda,\rho'}}\right)^{1/k} \leq \frac{\limsup_{k\to\infty} (C_k^{\lambda,\rho})^{1/k}}{\limsup_{k\to\infty}(C_k^{\lambda,\rho'})^{1/k}} \leq \limsup_{k\to\infty} \left(\frac{C_k^{\lambda,\rho}}{C_k^{\lambda,\rho'}}\right)^{1/k}.$$
%%
Using the bound in \thref{lem:cat_close} (ii) results in 
$$1- \epsilon' \leq \f{M(\rho')}{M(\rho)} \leq 1+ \epsilon'.$$ Simplifying, then replacing out $\epsilon'$ gives $|M(\rho')-M(\rho)|<\epsilon$. Thus, $M$ is continuous at $\rho > 0$. Continuity for $\rho =0$ uses a similar argument along with \thref{lem:cat_close} (iii). The explicit formula for $M(0)$ comes from the formula $C_k^{\lambda, \rho} = C_k ( \lambda / (1+\lambda)^2)^k$ and the fact that the generating function for the Catalan numbers has radius of convergence $1/4$. 
\end{proof}

\begin{lemma} \thlabel{lem:rho_d}
If $\lambda \in \Lambda$ then there exists a unique value $\rho_d >0$ such that $M(\rho_d) = d$; Moreover, if $\lambda \notin \Lambda$, then $M>d$ for all $\rho \geq 0$. 
\end{lemma}
\begin{proof}
Fix $\lambda$. To signify the dependence on $\rho$, let $g_\rho(z)$ be the generating function of the $C_k^{\lambda,\rho}$. Using the continuity and strict monotonicity of $M$ in \thref{prop:M_cont}, to show the first statement, it suffices to prove that $g_\rho(d) < \infty$ for $\rho$ large enough, and that $g_0(d)=\infty$. It is easy to see that if $\rho>\lambda(d-1)$ then the branching process of red spreading with no blue particles has finite expected size, and thus $g_\rho(d)<\infty$ for such $\rho$.

Using the formula for $M(0)$ from \thref{prop:M_cont} and rearranging, we can see that when $\rho = 0$, $M <d$ whenever $\lambda d/(1+\lambda^2) > 1/4$. The set of $\lambda$ for which this occurs is by definition $\Lambda$. Therefore, $g_0(d) = \infty$, proving the first claim. 
The claim that $M>d$ for $\lambda \notin \Lambda$, follows from \thref{prop:M_cont}. The explicit formula for $M(0)$ is easily shown to satisfy $M(0)> d$, and since $M$ is increasing, this inequality holds for all $\rho \geq 0$. 
\end{proof}

%We remind the reader of a classical result for singularities of analytic functions. 

Our next lemma requires a old theorem from complex variable theory (see \cite[Theorem IV.6]{flajolet2009analytic} for example). 

\begin{thm}[Pringsheim's Theorem] \thlabel{thm:pringsheim}
	If $f(z)$ is representable at the origin by a series expansion that has non-negative coefficients and radius of convergence $M$, then the point $z = M$ is a singularity of $f(z)$.
\end{thm}

\begin{lemma}\thlabel{lem:M_regimes}
$M\leq d$ if and only if $g(d) = \infty$. 
\end{lemma}

\begin{proof}
We first note that the implication ``$M < d$ implies $g(d) = \infty$" as well as the reverse direction ``$g(d) = \infty$ implies $M \leq d$" both follow immediately from the definition of the radius of convergence. 
It remains to show that $M=d$ implies $g(d) = \infty$.
\thref{lem:meromorphic} proves that $f$ is a meromorphic function. Since $g=f$ for $|z| <M$, $f(x) > 0$ for $x \in (0,d)$, and \thref{thm:pringsheim} gives $z=d$ is a pole, we have for $x \in \mathbb R$ that $$g(d) = \lim_{x\uparrow d} g(x) = \lim_{x\uparrow d} f(x) = \infty.$$

\begin{comment}
Since all of the poles of $f$ have order at least one and $z=d$ is a pole by \thref{thm:pringsheim}, we see that $C_k^{\lambda,\rho_c}d^k \geq \epsilon >0$ for some $\epsilon >0$ and for infinitely many $k$. Thus, $f(d) = g(d)= \infty$. 
\end{comment}
%\HOX{I think we meant to remove this? -EB}
\end{proof}

\begin{remark} \thlabel{rem:EZ}
\emph{Why does the argument that $g(d) = \infty$ when $M=d$ not also apply to CE which we know has different behavior for $\lambda = \lambda_c^-$? The difference is that when $\rho=0$, $$g(z) = \sum_{k=0}^\infty C_{k}^{\lambda,0} = \sum_{k=0}^\infty C_k \left (\f \lambda {(1+\lambda)^2}\right)^k z^k = \f{1 - \sqrt{1- 4z\lambda / (1+\lambda)^2}}{2z}$$ has a branch cut rather than isolated poles. \cite[Theorem IV.10]{flajolet2009analytic} ensures that the growth of the $C_k^{\lambda,\rho_c}$ is determined by the orders of the singularities of $f$ at distance $M=d$. Formally,
\begin{align}
C_{k}^{\lambda,\rho_c} = \sum_{|\alpha_j| = d} \alpha_j^{-k} \pi_j(k)  + o(d^{-k})\label{eq:ac}
\end{align}
where the $\alpha_j$ are the poles of $f$ at distance $d$ and the $\pi_j$ are polynomials with degree equal to the order of the pole of $f$ at $\alpha_j$ minus one. 
 When evaluating at the radius of convergence $M=\lambda_c^-$, this gives a summable pre-factor of $k^{-3/2}$ in \eqref{eq:ac}. }
\end{remark}

\section{$M$ and CED} \label{sec:M}

%Recall that $M$ is the radius of convergence of the generating function $g(z) = \sum_{k\geq 0} C_k^{\lambda,\rho} z^k$. 
The main results from this section are that for $\lambda \in \Lambda$ we have $\rho_c = \rho_d$ and that $P(B)$ is continuous at $\rho_d$. The lemmas in this section will also be useful for characterizing the phase behavior of CED. 

\begin{prop} \thlabel{cor:rho_d=rho_c}
	For $\lambda \in \Lambda$ it holds that $\rho_d = \rho_c$.
\end{prop}
\begin{proof}
Lemmas \ref{prop:M<} and \ref{prop:M=} give that whenever $\lambda \in \Lambda$ we have $\rho_d = \inf\{ \rho \colon P_{\lambda ,\rho} (B) = 0\}$, which is the definition of $\rho_c$.  
\end{proof}

\begin{lemma} \thlabel{prop:M>}
$M>d$ if and only if $E| \B| < \infty$.
\end{lemma}

\begin{proof}
Letting $Y$ be as at \eqref{eq:Ydef}, we first observe that 
$$E|\B| = \sum_{k=0}^\infty P(Y=k) d^k \geq \sum_{k=1}^\infty P(\RR_k) d^k = g(d).$$
Hence $E|\B|<\infty$ implies $g(d)<\infty$, which gives $M >d$ by \thref{lem:M_regimes}. 

For the other direction, using the comparison in \thref{lem:bigger} gives
\begin{align}
E |\B| &= \sum_{k=0}^\infty P(Y=k) d^k 
\leq 1+ \sum_{k=1}^ \infty C k^{1+\lambda/\rho}P(\RR_k) d^k\label{eq:EB}.
\end{align}
\thref{prop:M_cont} tells us that $M$ is continuous and since $M>d$, the sum on the right still converges with the polynomial prefactor. 
\end{proof}

Call a vertex $v \in \T_d$ a \emph{tree renewal vertex} if at some time it is red, the parent vertex is blue, and all vertices  one level or more down from $v$ are white. Note that this definition of renewal is a translation one unit left of the definition of renewals in the CED$^+$. This makes it so each vertex at distance $k$ from the root is a tree renewal vertex with probability $P(\RR_{k})$.
We call a tree renewal vertex $v$ a \emph{first tree renewal vertex} if none of the non-root vertices in the shortest path from $v$ to the root are renewal vertices. 
Let $Z$ be the number of first tree renewal vertices in CED.

\begin{lemma}\thlabel{lem:bp}
$E Z \geq 1$ if and only if $g(d) = \infty$. 
\end{lemma}

\begin{proof}
 Using self-similarity of $\T_d$, the number of renewal vertices is a Galton-Watson process with offspring distribution $Z$. Since each of the $d^k$ vertices at distance $k$ have probability $P(\RR_k)$ of being a tree renewal vertex, the expected size of the Galton-Watson process is $g(d)$. Standard facts about branching processes imply the claimed equivalence. 
\end{proof}

 \begin{lemma} \thlabel{lem:M=d}
If $\lambda \in \Lambda$ and $\rho=\rho_d$, then $EZ\leq 1$. 	
\end{lemma}

\begin{proof}
 Since $M$ is strictly increasing in $\rho$ (\thref{prop:M_cont}) we have $M>d$ for $\rho>\rho_d$ and, so,  \thref{lem:bp} implies that $E Z < 1$. Let $\RR_k'$ be the event that a fixed, but arbitrary vertex at distance $k$ is a first tree renewal vertex. Fatou's lemma gives
\begin{align}
E_{\lambda,\rho_d} Z = \sum_{k=1}^\infty d^k P_{\lambda,\rho_d}(\RR'_k)&=\sum_{k=1}^\infty d^k \liminf_{\rho \downarrow \rho_d} \P(\RR_k') \label{eq:R'}\\
		& \leq \liminf_{\rho \downarrow \rho_d} \sum_{k=1}^\infty d^k \P(\RR_k') 
		= \liminf_{\rho \downarrow \rho_d} E_{\lambda, \rho} Z 
		\leq 1.
\end{align}
To see the second equality, notice that $\P(\RR_k')$ is continuous in $\rho$ at $\rho_d$, as $\RR_k'$ consists of finitely many jump chains.  
 
\end{proof}

\begin{lemma} \thlabel{prop:M<}
If $\lambda \in \Lambda$ and $\rho< \rho_d$, then $P(B) > 0$.\end{lemma}

\begin{proof}
%If $M>d$, then we have $\rho < \rho_d$ from \thref{lem:rho_d}. 
It follows from \thref{prop:M_cont,lem:bp}, along with the easily observed fact that $EZ$ is strictly decreasing in $\rho$ that $EZ>1$ for $\rho < \rho_d$. Thus, the embedded branching process of renewal vertices is infinite with positive probability. As $|\B|$ is at least as large as the number of renewal vertices, this gives $P(B) >0$.
\end{proof}

\begin{lemma}\thlabel{prop:M=}
If $\lambda \in \Lambda$ and $\rho = \rho_d$, then $P(B) = 0$. 
\end{lemma}

\begin{proof}
%By comparing to the embedded branching process, we have $E |\B| = \infty$. 
The probability blue survives along a path while remaining at least distance $L >0$ from the most distant red site from the root for $k$ consecutive steps is bounded by the probability none of the intermediate red particles die, a quantity which is in turn bounded by:
\begin{align}
\left(\frac{1+\lambda}{1+ \lambda + L \rho}\right)^{k}. \label{eq:L}
\end{align}
Choose $L_0$ large enough so that \eqref{eq:L} is smaller than $(2d)^{-k}$. Let $D_k(v)$ be the event that, in the subtree rooted at $v$, the front blue particle moves from vertex $v$ to a vertex at level $k$ without ever coming closer than $L_0$ to the front red particle. Applying a union bound over all vertices at distance $k$ from $v$ and the estimate \eqref{eq:L} we have $P(\cup_{|v|=k} D_k(v)) \leq 2^{-k}$. As these probabilities are summable, the Borel-Cantelli lemma implies that almost surely only finitely many $D_k$ occur. 

On a given vertex self-avoiding path from the root to $\infty$, let $B_t$ (respectively, $R_t$) be the distance between the furthest blue (respectively, the furthest red) and the root. In order for $B$ to occur there must exist a path such that $B_t - R_t =m$ infinitely often for some fixed $m < L_0$. Suppose $B_t - R_t \geq m$ for all times and $B_t - R_t = m$ infinitely often. Self-similarity ensures that the vertices at which $B_t - R_t = m$ form a branching process. Using monotonicity of the jump chain probabilities at \eqref{eq:transitions}, this branching process is dominated by the embedded branching process of renewal vertices (i.e., when $m = 1$). \thref{lem:M=d} ensures that the embedded branching process of renewal vertices is almost surely finite (since it is critical), and thus for each fixed $m < L_0$ the associated branching process is also almost surely finite. As blue can only reach infinitely many sites if either infinitely many $D_k$ occur, or one of the $m$-renewing branching processes survives, we have $P(B)=0$. 
\end{proof}

\section{Proofs of Theorems \ref{thm:main}, \ref{cor:EB}, and \ref{cor:growth}} \label{sec:proofs}

\begin{proof}[Proof of \thref{thm:main}]
First we prove that $\rho_e = \lambda(d-1)$. In CED on $\T_d$ with no blue particles present, red spreads as a branching process with mean offspring $d \lambda / (\lambda + \rho)$. This is supercritical whenever $\rho < \lambda (d-1)$. It is easy to see that $P(R)>0$ for such $\rho$, since with positive probability a child of the root becomes red, then the red particle at the root dies. At this point, blue cannot spread, and red spreads like an unchased branching process. Since the red branching process is not supercritical for $\rho \geq \lambda(d-1)$, we have $P(R) =0$ for such $\rho$, which proves that $\rho_e$ is as claimed.

Now we prove (i). Suppose that $\lambda  \in \Lambda$. 
\thref{lem:rho_d} and \thref{cor:rho_d=rho_c} ensure that $\rho_c>0$, and the bound at \eqref{eq:f1f2} in \thref{cor:growth} implies $\rho_c< \lambda(d-1) = \rho_e$.
For $0\leq \rho < \rho_c$, \thref{prop:M<} implies that coexistence occurs with positive probability. For $\rho_c \leq \rho < \rho_e$, \thref{prop:M=} and \thref{prop:M<} imply that $P(B) =0$, but the red branching process survives so $P(R) >0$. Thus, escape occurs. For $\rho \geq \rho_e$, red cannot survive in CED with no blue, so extinction occurs.

We end by proving (ii). Suppose that $\lambda \notin \Lambda$. By \thref{lem:rho_d} we have $M>d$ for all $\rho \geq 0$. It then follows from \thref{prop:M>} that $E|\B| <\infty$ and so $P(B) =0$ for all such $\rho$. 
Similar arguments as in (i) that only consider the behavior of red after it separates from blue can show that the escape and extinction phases occur for $0\leq \rho < \rho_e$ and $\rho \geq \rho_e$, respectively.
\end{proof}

\begin{proof}[Proof of \thref{cor:EB}]
\thref{prop:M_cont} and \thref{cor:rho_d=rho_c} give that $\rho>\rho_c$ implies $M>d$, which, by \thref{prop:M>}, implies that $E|\B| < \infty$. This gives (i). The claim at (ii) holds because \thref{lem:rho_d} and \thref{cor:rho_d=rho_c} imply that $M=d$ when $\rho = \rho_c$. \thref{lem:M_regimes} gives that $M=d$ implies $g(d) = \infty$. Since $E |\B| \geq g(d)$, we obtain the claimed behavior. 
\end{proof}

\begin{proof}[Proof of \thref{cor:growth}]
Define the functions \begin{align*}
f_1(\lambda) &= \f 1{4}\left( \sqrt{(8d + 2) \lambda + \lambda^2 + 1} - 3 \lambda -3 \right), \\
f_2(\lambda)&= \frac{1}{4}\left(   \sqrt{ (32 d +2 ) \lambda + \lambda^2 + 1} -3 \lambda - 3 \right).
\end{align*}
We will prove that 
\begin{align}f_1(\lambda) \leq \rho_c(\lambda) < f_2(\lambda) \label{eq:f1f2}.
	\end{align}
Upon establishing this, it follows immediately that for large enough $d$ we have $c \sqrt d \leq \rho_c \leq C \sqrt d$ for any $c< \sqrt{\lambda /2}$ and $C > \sqrt{2 \lambda}.$

On a given path from the root, the probability that red advances one vertex and then blue advances is given by 
\begin{align} p = \f{\lambda}{(1+ \lambda + \rho)(1+\lambda + 2 \rho)}. \label{eq:p}
\end{align}

For each vertex $v$ at distance $k$ from the root, let $G_v$ be the event that red and blue alternate advancing on the path from the root to the parent vertex of $v$, after which red advances to $v$, and then does not spread to any children of $v$ before the parent of $v$ is colored blue. Letting  
\begin{align}c=\f{ \lambda}{1+ \lambda + \rho} \f{ 1}{ 1+ d \lambda + 2\rho},\label{eq:c}	
\end{align}
it is easy to see that $P(G_v) = cp^{k-1}$. A renewal occurs at each $v$ for which $G_v$ occurs. 
It is straightforward to verify that $pd>1$ whenever $\rho < f_1(\lambda)$. Accordingly, we can choose $K$ large enough so that $cp^{K-1} d^{K}>1$, and thus the branching process of these renewals is infinite with positive probability, which implies that $P(B) >0$. 

To prove the upper bound we observe that monotonicity of the transition probabilities at \eqref{eq:transitions} in $j$ ensures that the maximal probability jump chain of length $2k$ is the sawtooth path that alternates between height $1$ and height $2$. This path occurs as a living jump chain with probability $p^k$ with $p$ as in \eqref{eq:p}.
As this is the maximal probability jump chain of the $C_k$ many Dyck paths counted by $\RR_k$, we have  $g(d) \leq \sum_{k=1}^\infty C_k p^k d^k.$
The radius of convergence for the generating function of the Catalan numbers is $1/4$ with convergence also holding at $1/4$. Thus, when $\rho$ is large enough that $pd \leq 1/4$, we have $g(d) <\infty$. \thref{lem:M_regimes} and \thref{prop:M>} then imply that $E|\B| < \infty$ and so $\rho > \rho_c$.
It is easily checked that the above inequality holds whenever $\rho \geq f_2(\lambda)$. Thus, $\rho_c < f_2(\lambda).$

One can further check that $f_2(\LM) =0 = f_2(\LP)$ are the only zeros of $f_2$, so that $f_2(\lambda)>0$ for $\lambda \in \Lambda$. 
Moreover, we have
$$\rho_e - f_2(\lambda) = \lambda(d-1) - \frac{1}{4}\left(   \sqrt{ (32 d +2 ) \lambda + \lambda^2 + 1} -3 \lambda - 3 \right).$$
Some algebra gives $\rho_e - f_2(\lambda)=0$
is equivalent to solving $-8+\lambda (8 d+8)+\lambda^2 (8 \lambda-16 \lambda^2)=0,$
which has no real solutions for $d\geq 2$. 
As $\rho_e - f_2(\lambda)$ is continuous in $\lambda$ and positive at any choice of $d\geq 2$ and $\lambda$, we must have $\rho_e - f_2(\lambda) >0$. We thus arrive at the claimed relation $\rho_c \leq f_2(\lambda) < \rho_e$. 
\end{proof}

%We next use our work from Section \ref{sec:CED+} to show that $F(\rho) <1$ implies $E|\B|$ is finite.

\section{Proof of \thref{thm:smooth}} \label{sec:smooth}

Recall our notation for continued fractions at \eqref{eq:K}. In this section we use the definition of $a_j$ at \eqref{eq:aj} and also define $b_j:= a_j d$. 
The following lemma shows that the nested continued fractions in the definition of $f$ at \eqref{eq:f} are decreasing when $z= M$. 

\begin{lemma} \thlabel{lem:K_mono}
$K[a_i M, a_{i+1}M, \hdots] \leq K[a_{i-1}M, a_iM, \hdots]$ for all $i \geq 1$. 	
\end{lemma}

\begin{proof}
Let $f_i(x) = K[a_i x, a_{i+1} x, \hdots]$. Since the $f_i$ are analytic for $x<M$, we must have $f_i(x) < 1 $ for all $x < M$ (otherwise $f$ would have a singularity with modulus smaller than $M$). For any fixed $n$ we can the use monotonicity of $K$ when we change the entries of $K$ one by one and use the fact that $a_j < a_{j-1}$ for all $j \geq 1$ to conclude that 
%\HOX{I added a comment on the monotonicity of $K$ with respect to each entry. -SH}
\begin{align}
K[a_i x, a_{i+1}x, \hdots, a_n x] 
%&< K[a_{i-1} x, a_{i+1} x, \hdots, a_{n}x] 
%\\
%&< K[a_{i-1} x, a_{i} x,a_{i+2} x, \hdots, a_{n}x] \\
%& \; \; \vdots \\
&\leq K[a_{i-1} x, a_{i} x,a_{i+2} x, \hdots, a_{n-1}x].
\end{align}
Taking the limit as $n\to \infty$ gives $f_i(x) \leq  f_{i-1}(x)$ for all $x < M$. Letting $x \uparrow M$, these inequalities continue to hold. 
 \end{proof}
Note that since $f(x)$ has a pole at $M$ and $f_i(M) \leq f_{i-1}(M)$, we must have that $f_0(M) = 1$ and $f_i(M) \leq 1$ for all $i \geq 0$.

\begin{proof}[Proof of \thref{thm:smooth}] 
When $\rho= \rho_c$, it follows from \thref{cor:rho_d=rho_c} and \thref{thm:pringsheim} that a pole of $f$ occurs at $z=d$. 
Let $K(\lambda,\rho) = K[b_0,b_1,\hdots]$. Due to \thref{lem:K_mono} and the equality $f(z) = (1- K[a_0 z, a_1 z,\hdots])^{-1},$
the singularity at $z=d$ occurs as a result of $K(\lambda, \rho_c(\lambda)) = 1.$ 

We can use a similar argument as in \thref{lem:meromorphic} to view $f$ as a meromorphic function in the complex variables $\lambda, \rho$ and $z$. 
Thus, when fixing $z=d$ and considering $\lambda$ and $\rho$ as nonnegative real numbers, the function $K(\lambda, \rho)= f_{\lambda,\rho}(d)$ is real analytic. 
Moreover, since $K$ is easily seen to be strictly decreasing in $\rho$, we have $\partial K/\partial \rho \neq 0$ at $(\lambda , \rho_c(\lambda))$. 
As $K(\lambda, \rho_c(\lambda)) \equiv 1,$
and $K$ is infinitely differentiable, it follows from the implicit function theorem that $\rho_c(\lambda)$ is smooth.
\end{proof}

\section{Proof of \thref{cor:algorithm}} \label{sec:approximations}

We begin this section by describing  lower and upper bounds on $C_k^{\lambda, \rho}$. These are easier to analyze than $C_k^{\lambda, \rho}$ and lend insight into the local behavior of $\rho_c$. In particular, we obtain if and only if conditions to have $\rho < \rho_c$ (\thref{lem:<iff}) and $\rho > \rho_c$ (\thref{lem:>iff}). We use these bounds to prove \thref{cor:algorithm}.

\subsection{A lower bound on $C^{\lambda,\rho}_k$}

 The idea is to assign weight 0 to rise and fall steps above a fixed height $m\geq 1$. Accordingly, we introduce the weights
	 \begin{align}\hat u(j) =  \begin{cases} u(j),& j \leq m \\ 0, &j > m
 		\end{cases}, \qquad \hat v(j) =  \begin{cases} v(j), &j \leq m \\ 0,& j> m
 		\end{cases}. \label{eq:hat_uv}	
 \end{align} 
 Here $u(j)$ and $v(j)$ are as in \eqref{eq:ud}. Let $\hat C_k^{\lambda,\rho}$ be the corresponding weighted Catalan numbers. Since $\hat u(j) \leq u(j)$ and $\hat v(j) \leq v(j)$ for all $j$, we have  $\hat C_k^{\lambda, \rho} \leq C_k^{\lambda, \rho}.$
 
 Let $\hat g_m(z) = \sum_{k=0}^\infty \hat C_k^{\lambda, \rho} z^k$. The truncation ensures that $\hat g_m(z) = K[1,a_0z,\hdots, a_m z]$
 is a rational function with radius of convergence $\hat M \geq M.$ Note that $\hat M$ depends on $\lambda, \rho$ and $m$. For nonnegative real values $x$ we have $\hat g_m(x) < \hat g_{m+1}(x)$. It follows from the monotone convergence theorem that $\lim_{m \to \infty} \hat g_m(x) = g(x)$
and thus 
\begin{align}\lim_{m \to \infty} \hat M = M.\label{eq:M_lim}\end{align}

\begin{lemma} \thlabel{lem:<iff}
$\rho < \rho_c$ if and only if $K[b_i,\hdots,b_m] >1$ for some $m\geq 1$ and $0 \leq i \leq m$.
%there exists $m\geq 2$ such that 
%\begin{align} K[b_0,\hdots, b_i] < 1 \text{ for all %$0\leq i \leq m$}. \label{eq:K_cond}	
%\end{align}
 
\end{lemma}

\begin{proof}
Suppose that $\rho < \rho_c$. \thref{prop:M_cont} and \thref{cor:rho_d=rho_c} imply that $M < d$. By \eqref{eq:M_lim}, we have $\hat M < d$ for a large enough choice of $m$. Since $\hat g_m(x)$ is a rational function, its singularities occur wherever one of the partial denominators $1 - K[a_i x, \hdots, a_m x] = 0$. Let $i^*$ be the largest index such that there is $x_0 <d$ with 
$$1- K[a_{i^*} x_0, \hdots,a_m x_0] = 0.$$
Since $i^*$ is the maximum index for which this holds, we have $K[a_{j}x,\hdots, a_mx] <1$ for all $j > i^*$ and $x \in (x_0,d)$. 
%\HOX{ Why is it true that $K[b_{j},\hdots, b_m] <1$? Is it true that we cannot have a singularity at $d$?
%I can see why $K[ a_j x ,\hdots, a_m x ] <1$ for $x \in (x_0, d)$, and that's enough for the conclusion. -SH }
This ensures that $K[a_{i^*} x, \hdots,a_m x]$ is a strictly increasing function for $x \in (x_0,d)$. Thus, $K[b_{i^*},\hdots, b_m] >1$. 

Suppose that $K[b_i,\hdots,b_m] >1$ for some $m\geq 1$ and $0 \leq i \leq m$. This implies that $\hat g_m$ has a singularity of modulus strictly less than $d$. Thus, $\hat M < d$. Since $M \leq \hat M$, \thref{prop:M<} implies that $\rho < \rho_c$. 
\end{proof}

\subsection{An upper bound on $C_k^{\lambda, \rho}$}
Since $u(j)$ and $v(j)$ are decreasing in $j$, we obtain an upper bound by assigning weight $u(m)$ and $v(m)$ to all rise and fall steps above height $m$. More precisely, set
 \begin{align}\tilde u(j) =  \begin{cases} u(j),& j < m \\ u(m), &j \geq m
 		\end{cases}, \qquad \tilde v(j) =  \begin{cases} v(j), &j < m \\ v(m),& j \geq m
 		\end{cases}. \label{eq:tilde_uv}	
 \end{align}

 Let $\tilde C_{k}^{\lambda ,\rho}$ be the weighted Catalan number using the weights $\tilde u$ and $\tilde v$. Also, let $\tilde g$ and $\tilde M$ be the corresponding generating function and radius of convergence. Since $u(j) \leq \tilde u(j)$ and $v(j) \leq \tilde v(j)$ for all $j \geq 0$, we have $C_{k}^{\lambda,\rho} \leq \tilde C_{k}^{\lambda ,\rho}$
for all $k\geq 0$. It follows that the corresponding generating functions and radii of convergence satisfy
$g(x) \leq \tilde g(x),$
and $M \geq \tilde M$.

There is a finite procedure for bounding $\tilde M$. We say that $K[1,c_0,c_1,\hdots,c_k]$ is \emph{good} if all of the partial continued fractions are smaller than $1$: 
\begin{align}
c_k &<1 \label{eq:good_test} , \quad  K[c_{k-1},c_k]  <1  ,  \quad \cdots \quad  , \quad K[c_0,\hdots, c_k]  <1.
\end{align}
%Otherwise we say that the continued fraction is \emph{bad}.
%Let 
	%\begin{align}
	%f_m(x) = K[1,a_mx,a_{m}x,\hdots]= \f{2}{1+ %\sqrt{1- 4 a_m x} } \label{eq:f_m}
%	\end{align}
%	be the generating function for the Catalan numbers evaluated at $a_mx$.
 Define the continued fraction
	\begin{align}
		K_m(x):=K[1,a_0x,\hdots, a_{m-2}x, a_{m-1}x \psi(a_m x)] \label{eq:Kx}
	\end{align}
	where $\psi(x) = K[1,x,x,\hdots]$ is the generating function of the usual Catalan numbers.
By \eqref{eq:f}, so long as $x < \tilde M$ we have 
$$\tilde g(x) = K_m(x).$$

We now prove that when a partial continued fraction is good, it is good in a neighborhood. 

\begin{comment}
\begin{lemma} \thlabel{lem:good}
If $K_m(d)$ is good, then there exists $\epsilon >0$ such that  for all $x_i \leq d(1+\epsilon)$ 
\begin{align}K[1,a_0x_0,\hdots,a_{m-2} x_{m-2}, a_{m-1} x_{m-1} \psi(a_m x_m)]\label{eq:Kxi}	
\end{align}
%
is good. 
\end{lemma}	

\begin{proof}
Let  $K_m(\vec x)$ denote the continued fraction at \eqref{eq:Kxi}. 
The $i$th partial continued fraction from $K_m(\vec x)$ that we check among those at \eqref{eq:good_test} is an increasing function of $x_i$.  
Thus, if  $K_m(x)$ is good, then each partial continued fraction from $K_m({\vec x})$ is also good for $x_i \leq d$. 
Since the inequalities are strict and $K_m({\vec x})$ is continuous in each variable, we can extend to have $K_m(\vec x_i)$ good for all $x_i \leq d(1+\epsilon)$ for some $\epsilon >0$. 
\end{proof}

\textbf{More general Lemma 22:}
\end{comment}

\begin{lemma}\thlabel{lem:good}
If $K[\mathbf{c}] := K[c_0, c_1, \dots, c_k]$ is good, then there exists an $\epsilon > 0$ such that $K[ \mathbf{\tilde{c}}] := K[\tilde{c}_0, \tilde{c}_1, \dots, \tilde{c}_k]$ is good when $\tilde{c}_j \leq c_j(1+\varepsilon)$ for all $j = 0, \dots, k$. 
\end{lemma}

\begin{proof}
If $K[ \mathbf{c}]$ is good, then each partial fraction in \eqref{eq:good_test} is $<1$. Note that each of those partial fractions is a decreasing function of the $c_j$ which appear in the fraction. Therefore, if $\tilde{c}_j \leq c_j$, $K[ \mathbf{\tilde{c}}]$ is good. Since the inequalities are strict and $K[\mathbf c]$ is continuous in each $c_j$, we can extend to have $K[\mathbf{ \tilde{c}}]$ good for all $\tilde{c}_j \leq (1+\epsilon)c_j$ where $\epsilon$ is chosen small enough so that none of the partial fractions in \eqref{eq:good_test} equal 1. 
\end{proof}

\begin{lemma} \thlabel{lem:>iff}
$\rho > \rho_c$ if and only if $b_m < 1/4$ and $K_m(d)$ is good for some $m \geq 1$. 
\end{lemma}

\begin{proof}
Suppose that $\rho > \rho_c$. Then \thref{prop:M_cont} and \thref{cor:rho_d=rho_c} say that $g$ has radius of convergence $M > d$. For $|z|<M$, we write 
\begin{align}
g(z) = K[1,a_0z , \hdots, a_{m-2} z, a_{m-1} z g_m(z)] \label{eq:good_g}
\end{align}
with $g_m(z):= K[1,a_m z, a_{m+1}z,\hdots]$
the tail of the continued fraction.

Because $M > d$, $g_m(d) < \infty$, we can see, using a quick argument by contradiction, that $g(d)$ is good. Suppose that one of the partial fractions in \eqref{eq:good_test} is strictly larger than $1$. But because these finite continued fractions are continuous in the inputs, this would imply that there is a singularity at some $|z| < d$, contradicting the fact that $M > d$. 

As the continued fraction representation of $g(d)$ is good, we then apply \thref{lem:good} to say that there exists an $\epsilon > 0$ such that 
\begin{align}
K[1,b_0 , \hdots, b_{m-2}, b_{m-1} g_m(d) (1+\epsilon)] \label{eq:gooder}
\end{align}
is also good.
Notice that $\psi(a_m x) \geq g_m(x) \geq 1$ by definition. Since $|a_{m}z| \to 0$ as $m \to \infty$, we can use the explicit formula for $\psi$ to directly verify that $\psi(a_m x) \to 1$ as $m \to \infty$. Choose $m$ large enough so that $b_m < 1/4$ and $\psi(b_m) < 1 + \epsilon$. Then $ \psi(b_m) < (1+\epsilon)g_m(d)$ and it follows from \eqref{eq:gooder} that $K_m(d)$ is good.

\begin{comment} %This is the old version of the proof. I left it in in case we want to revert back to it/some of it... 
Suppose that $\rho > \rho_c$. \thref{prop:M>} and \thref{cor:rho_d=rho_c} give that $\rho > \rho_c$ implies that $g$ has radius of convergence $M >d$.
For $|z|<M$, we write 
\begin{align}
g(z) = K[1,a_0z , \hdots, a_{m-2} z, a_{m-1} z g_m(z)] \label{eq:good_g}
\end{align}
with $g_m(z):= K[1,a_m z, a_{m+1}z,\hdots]$
the tail of the continued fraction.  Choose $m$ such that $b_m <1/4$. Because $M>d$, we know that $g$ has no singularities, and thus \eqref{eq:good_g} evaluated at $z=d$ is good. If it we not good, then one of the partial continued fractions equal $1$ for some $|z| < d$, which would cause a singularity and contradict that $M >d$. Moreover, \thref{lem:good} implies that there exists an $\epsilon>0$ such that
\begin{align}
K[1,b_0 , \hdots, b_{m-2}, b_{m-1} g_m(d) (1+\epsilon)] \label{eq:gooder}
\end{align}
is also good.

Notice that $\psi(a_m x) \geq g_m(x) \geq 1$ by definition. Since $|a_{m}z| \to 0$ as $m \to \infty$, we can use the explicit formula for $\psi$ to directly verify that $\psi(a_m x) \to 1$ as $m \to \infty$. 
Choose $m$ large enough so that $b_m < 1/4$ and $\psi(b_m) < 1 + \epsilon$.   
It follows from \eqref{eq:gooder} that $K_m(d)$ is good.
\end{comment}

Now, suppose that $b_m < 1/4$ and $K_m(d)$ is good for some $m \geq 1$. Our definitions of $\tilde u$ and $\tilde v$ ensure that we can write $$K[1,\tilde a_m x, \tilde a_{m+1} x,\hdots] = K[1, a_m x, a_mx, \hdots]= \psi(a_m x)$$ for all $x$ with $a_m x < 1/4.$ 
Since $b_m = a_m d < 1/4$, this ensures that this is true for all $x < d(1+\epsilon')$ for some $\epsilon' >0$. 
Similar reasoning as \thref{lem:meromorphic} then gives $K_z$ is a meromorphic function for $|z| \leq d(1+\epsilon')$. 
Moreover, \thref{thm:pringsheim} ensures that the first pole occurs at the smallest $x$ for which some partial continued fraction of $K_m$ is equal to $1$. 
By \thref{lem:good}, $K_m(d)$ being good implies that there exists $0< \epsilon\leq \epsilon'$ such that $K_m$ is good for all $x \leq d(1+\epsilon)$. 
Hence, there are no poles of $K_m$ within distance $d(1+\epsilon)$ of the origin.
 It follows that $\tilde g = K_m$ for all such $x$, and thus $\tilde M \geq d(1+\epsilon)$. Since $M \geq \tilde M$, we have $M > d$. This gives $\rho > \rho_c$ by \thref{prop:M>} and \thref{cor:rho_d=rho_c}.

\end{proof}

\subsection{A finite runtime algorithm}

\begin{proof}[Proof of \thref{cor:algorithm}]
Suppose we are given $\rho \neq \rho_c$. Lemmas \ref{lem:<iff} and \ref{lem:>iff} give a finite set of conditions to check whether $\rho < \rho_c$ or $\rho > \rho_c$, respectively. To decide which is true, we increase $m$ and the algorithm terminates once the conditions holds. 
\end{proof}

\subsection*{Acknowledgements}
Thanks to David Sivakoff and Joshua Cruz for helpful advice and feedback. We are also grateful to Sam Francis Hopkins for pointing us to a reference about weighted Catalan numbers. Feedback during the review process greatly improved the final version. This work was partially supported by NSF DMS Grant \#1641020 and was initiated during the  2019 AMS Mathematical Research Community in Stochastic Spatial Systems.

\bibliographystyle{alpha}
\bibliography{cewd}

\end{document}